\documentclass[runningheads]{llncs}
\usepackage{amssymb}
\usepackage{amsmath}
\usepackage{latexsym}
\usepackage{verbatim}
\usepackage{amsfonts}%
\usepackage{graphicx}
\usepackage{latexsym,graphicx}
\usepackage[all]{xy}
\usepackage{xcolor}
\usepackage{multirow}
\usepackage{todonotes}
\usepackage{ulem}
\usepackage{enumerate}
\usepackage{hyperref}
\usepackage{algorithm}
\usepackage{algorithmic}

\newcommand{\ord}{{\mathbf{ord}}}
\newcommand{\nat}{\bbbn}

\newcommand{\ra}{\rightarrow}

\newtheorem{Lemma}{Lemma}

\begin{document}
\mainmatter
\title{Unary Patterns of Size Four with Morphic Permutations}

\titlerunning{Patterns of Size Four with Morphic Permutations}

\tocauthor{Kamellia Reshadi}

\author{Kamellia Reshadi}

\authorrunning{K. Reshadi}

\institute{Institut f{\"u}r Informatik, Christian-Albrechts-Universit{\"a}t zu Kiel.\\
\email{kre@informatik.uni-kiel.de}
}

\maketitle
\begin{abstract}
We investigate the avoidability of unary patterns of size of four with morphic permutations. More precisely, we show that, for the positive integers $i,j,k$, the sizes of the alphabets over which a pattern $x \pi ^ {i} (x) \pi^{j}(x) \pi^{k}(x)$ is avoidable are an interval of the integers (where $x$ is a word variable and $\pi$ is a function variable with values in the set of all morphic permutations of the respective alphabets). We also show how to compute a good approximation of this interval.
This continues the work of  [Manea et al., 2015], where a complete characterisation of the avoidability of cubic patterns with permutations was given.
\end{abstract}

\section{Introduction}
The avoidability of patterns in infinite words is an old area of interest with a~first systematic study going back to Thue~\cite{Thue:06,Thue:12}. In these initial papers it was shown that there exist a binary infinite morphic word and a ternary infinite morphic word that avoid cubes and squares, respectively. That is, these infinite words do not contain instances of the patterns $xxx$ and $xx$, respectively. The most important classical results on avoidability are surveyed in several chapters of \cite{Loth97}; see, e.g., Chapters 2 and 7 of \cite{Loth97} and the references therein. 

In this article, we are studying the avoidability of repetitions in a generalised setting. Namely, we are interested in the avoidability of unary patterns with functional dependencies between variables. We are considering patterns like $x \pi^i(x)\pi^j(x)\pi^k(x),$ where $x$ is a word variable while $\pi$ is function variable, which can be replaced by bijective morphisms only.
The instances of such patterns over an alphabet $\Sigma$ are obtained by replacing $x$ with a concrete word, and $\pi$ by a morphic permutation of $\Sigma$. For example, an instance of the pattern $x\pi(x)x\pi(x)$ over $\Sigma=\{a,b\}$ is the word $uvuv$ such that $|u|=|v|,$ and $v$ is the image of $u$ under any permutation on the alphabet. Considering the permutation $a \to b$, and $b \to a$, then $aba|bab|aba|bab$ is an instance of $x \pi(x) x \pi(x).$ 

In this setting, we continue the work of~\cite{DLT2015}  and~\cite{MMN12} as follows. In~\cite{MMN12}, a complete characterisation of the avoidability of cubic patterns with permutations $x\pi^i(x)\pi^j(x)$ was given. Furthermore, in~\cite{DLT2015}  it was shown  that  there exists  a ternary word that avoids all patterns $\pi_{i_1}(x)\ldots \pi_{i_r}(x)$ where $r\geq 4$, $x$ a word variable over some alphabet $\Sigma$, with $|x|\geq 2$ and $|\Sigma|\geq 3$, and the $\pi_{i_j}$ function variables that may be replaced by anti-/morphic permutations of $\Sigma$. However, this result only holds when the length of $x$ is restricted to be at least~$2$. Also, in an extension \cite{TCS2018} of the aforementioned paper \cite{DLT2015}, it was shown that all patterns $\pi^{i_1}(x)\ldots \pi^{i_n}(x)$ with $n\geq 4$ under morphic permutations are avoidable in alphabets of size $2,3,$ and $4$, but there exist patterns which are unavoidable in alphabets of size $5$. We extend these results by showing how to determine exactly, for a given unary pattern ${\mathcal P}$ of size four with permutations, which are the alphabets in which it is avoidable.

The main result of our paper is that given $i,j,k \geq 0$, we show how to compute the value $\sigma$ such that the pattern $x \pi^{i}(x) \pi^{j}(x) \pi^{k}(x)$ is unavoidable in alphabets of size at least $\sigma+1$ and avoidable in alphabets of size $2,3,4,\ldots, \sigma-1$. The avoidability of this pattern in alphabets of size $\sigma$ has to be analysed individually for some $i,j,k$. Acoordingly, we show that for each pattern there exists an interval (whose left end is $2$ and right end is defined based on the respective parameters) such that over each alphabet whose size is in the respective interval, there exists an infinite word that does not contain instances of the given pattern. This shows that the main result of~\cite{MMN12} holds in the more general case of unary patterns of size four. However, the technicalities we develop here are much more involved.

The structure of the paper is as follows: we first give a series of basic definitions and preliminary results. Then we define the aforementioned parameters, and show how to use them to compute, for a given pattern $p$, the value $\sigma$ such that $p$ is unavoidable over alphabets with $m>\sigma $ letters. Finally, we show the dual of the previous result: for alphabets with at most $\sigma-1$ symbols the pattern $p$ is avoidable. Due to space constraints, most of the repetitive technicalities of this paper (e.g., a list of infinite words avoiding certain patterns) are given in the Appendix. However, the main part contains the major ideas needed to obtain the results we state. 

\section{Preliminaries}
We define $\Sigma_k = \left\{ 0, \dots, k-1 \right\}$ to be an alphabet with $k$ letters; the empty word is denoted by $\varepsilon$.
For words $u$ and $w$, we say that $u$ is a prefix (resp. suffix) of $w$, if there exists a word $v$ such that $w = uv$ (resp. $w = vu$).
If $f:\Sigma_k \ra \Sigma_k$ is a permutation, we say that the order of $f$, denoted $\ord(f)$, is the minimum value $m>0$ such that $f^m$ is the identity. If $a\in \Sigma_k$ is a letter, the order of $a$ with respect to $f$, denoted $\ord_f(a)$, is the minimum number $m$ such that $f^m(a)=a$. A function $f:\Sigma_k^* \ra \Sigma_k^*$ is a morphism if $f(xy)=f(x)f(y)$ for all words $x,y$; $f$~is a morphic permutation if the restriction of $f$ to $\Sigma_k$ is a permutation of $\Sigma_k$.

A pattern with functional dependencies is a term over (word) variables and function variables (where concatenation is an implicit functional constant). For example, $x \pi(y) \pi(\pi(x)) y$ is a pattern involving the variables $x$ and $y$ and the function variable $\pi$. An instance of a pattern $p$ in $\Sigma_k$ is the result of substituting uniformly every variable by a word in $\Sigma_k^+$ and every function variable by a function over $\Sigma_k^*$. A pattern is avoidable in $\Sigma_k$ if there is an infinite word over $\Sigma_k$ that does not contain any instance of the pattern. 

In this paper, we consider only unary patterns (i.e., containing only one variable) with morphic permutations, that is, all function variables are unary and are substituted by morphic  permutations only. 

The infinite Thue-Morse word $t$ is defined as 
$t = \lim_{n \to \infty} \phi_t^n(0),$ 
for the morphism $\phi_t : \Sigma_2^* \to \Sigma_2^*$ where $\phi_t(0)=01$ and $\phi_t(1)=10$. It is well-known (see \cite{Loth97}) that the word $t$ avoids the patterns $xxx$ (cubes) and $xyxyx$ (overlaps).

The infinite ternary Thue word $h$ is defined as
$h = \lim_{n \to \infty} \phi_h^n(0),$ 
for the morphism $\phi_h : \Sigma_3^* \to \Sigma_3^*$ where $\phi_h(0)=012$, $\phi_h(1)=02$ and $\phi_h(2)=1$. 
The infinite word $h$ avoids the pattern $xx$ (squares).



This paper is related to the study of the avoidability of cubic patterns with permutations from~\cite{MMN12}. In the respective paper for a given pattern $x\pi^{i}(x)\pi^{j}(x)$ the authors defined the following four values: $\alpha_1 = \inf \{ t: t \nmid |i-j|, t \nmid i, t \nmid j \}$, $\alpha_2 = \inf \{ t: t \mid |i-j|, t \nmid i, t \nmid j \}, \alpha_3= \inf \{ t: t \mid i, t \nmid j \}, \alpha_4= \inf \{ t: t \nmid i, t \mid j \}.$ Further, for $k=\min \{ \max \{\alpha_1,\alpha_2\}, \max \{\alpha_1,\alpha_3\}, \max \{\alpha_1, \alpha_4 \}\}$, it was shown that  $x\pi^{i}(x)\pi^{j}(x)$ is unavoidable in $\Sigma_m$, for $m \geq k$, and avoidable in $\Sigma_m$, for $4\leq m<k$. The avoidability of  $x\pi^{i}(x)\pi^{j}(x)$ in $\Sigma_2$ and $\Sigma_3$ was separately investigated, and a complete characterisation of the alphabets over which a pattern  $x\pi^{i}(x)\pi^{j}(x)$ is avoidable was obtained.

The reader is referred to \cite{Loth97,MMN12,TCS2018} for  further details. 
All computer programs referenced in this paper can be found at \url{http://media.informatik.uni-kiel.de/zs/patterns.zip}.

\section{Avoidability of patterns under permutations}
In this section we try to identify an upper bound on the size of the alphabets $\Sigma_m$ in which a pattern $x \pi^{i}(x) \pi^{j}(x) \pi^{k}(x)$, with $i, j, k \geq 0$ is unavoidable, when $\pi$ is substituted by a morphic permutation. 

In the pattern $x\pi^i{(x)}\pi^j{(x)}\pi^k{(x)}$, the factors $x$, $\pi^i(x)$, $\pi^j(x)$, or $\pi^k(x)$ are called $x$-items in the following. Our analysis is based on the relation between the possible images of the four $x$-items occurring in a pattern, following the ideas of \cite{MMN12}. For instance, we want to check whether in a possible image of our pattern, all four $x$-items can be mapped to a different word, or whether the second and the last $x$-items can be mapped to the same word, etc.

To achieve this, we define in Table \ref{table_k} the parameters $\alpha_a$, with $1\leq a \leq 14$. Intuitively, they allow us to define, for a pattern $x\pi^i{(x)}\pi^j{(x)}\pi^k{(x)}$, which are the alphabets $\Sigma_m$ in which we can model certain (in-)equality relationships between the images of the $x$-items. For example, in alphabets $\Sigma_m$ with $m\geq \alpha_1$ we can assign values to $x$ and $\pi$ such that the images of every two of $\pi^i{(x)}$, $\pi^j{(x)}$, and $\pi^k{(x)}$ are different (and this property does not hold in alphabets with less than $\alpha_1$ letters). Also, in $\Sigma_m$ with $m\geq \alpha_2$ we can assign values to $x$ and $\pi$ such that the images of $x$ and $\pi^i(x)$ are equal to some word, while the images of $\pi^j(x)$ and $\pi^k(x)$ are assigned to two other distinct words (also different between them; again, this property does not hold in smaller alphabets). To simplify, we use a simple digit-representation for any of these cases, defined in the last column of Table \ref{table_k}. In this representation of each $\alpha_a$, we assign different digits to the $x$-items that can be mapped to different words in alphabets of size at least $\alpha_a$. For example, we use the representation 0123 for the case defined through $\alpha_1$ and 0012 for the case defined by $\alpha_2$. In general, when considering an $\alpha_a$, we assign a $4$-digit representation to the pattern  $x\pi^i{(x)}\pi^j{(x)}\pi^k{(x)}$ in the following manner: we start with $0$, and then put a $0$ on all of the remaining three positions corresponding to an $x$-item $\pi^t(x)$ to such that $\alpha_a$ divides $t$. We then put a $1$ on the the leftmost empty position. If the $x$-item on the respective position is $\pi^r(x)$, we put $1$ on all empty positions $s$ such that $\alpha_a$ divides $(r-s)$, and so on. 

Please note that the actual values the parameters $\alpha_a$, with $1\leq a\leq 14$, take depend on the pattern $x\pi^i{(x)}\pi^j{(x)}\pi^k{(x)}$, and, more precisely, on $i,j,k$. Thus, for different patterns we will have different parameters.

\begin{table}
\begin{center}
  \begin{tabular}{| l | c|} 
    \hline
    $\alpha_1 = \inf \{ t: t \nmid i, t \nmid j, t \nmid k, t \nmid |i-j|, t \nmid |i-k|, t \nmid |j-k|\}$ & 0123 \\ \hline
    $\alpha_2 = \inf \{ t: t \mid i, t \nmid j, t \nmid k,  t \nmid |j-k|\}$ & 0012 \\ \hline
    $\alpha_3 = \inf \{ t: t \nmid i, t \mid j, t \nmid k, t \nmid |i-k|\}$ & 0102 \\ \hline
    $\alpha_4 = \inf \{ t: t \nmid i, t \nmid j, t \mid |i-k|\}$ & 0121 \\ \hline
    $\alpha_5 = \inf \{ t: t \nmid i, t \nmid j, t \nmid |i-j|, t \nmid |i-k|, t \mid |j-k|\}$ & 0122 \\ \hline
    $\alpha_6 = \inf \{ t: t \mid i, t \mid j, t \nmid k\}$ & 0001 \\ \hline
    $\alpha_7 = \inf \{ t: t \mid i, t \nmid j, t \mid k\}$ & 0010 \\ \hline
    $\alpha_8 = \inf \{ t: t \nmid i, t \mid j, t \mid k\}$ & 0100 \\ \hline
    $\alpha_9 = \inf \{ t: t \nmid i, t \mid|i-j|, t \mid |i-k| \}$ &  0111 \\ \hline
    $\alpha_{10} = \inf \{ t: t \mid i,  t \nmid j, t \mid|j-k| \}$ & 0011 \\ \hline
    $\alpha_{11} =  \inf \{ t: t \nmid i,  t \mid j, t \mid|i-k| \}$ & 0101 \\ \hline
    $\alpha_{12} =  \inf \{ t: t \nmid i,  t \mid k, t \mid|i-j| \}$ & 0110 \\ \hline
    $\alpha_{13} =  \inf \{ t: t \nmid i,  t \nmid k, t \mid|i-j| \}$ & 0112 \\ \hline
    $\alpha_{14} =   \inf \{ t: t \nmid i,  t \nmid j, t \mid|i-j| \}$ & 0120 \\ \hline
  \end{tabular}
\end{center}
\caption{Definition of the values $\alpha_a$, with $1\leq a\leq 14$. }
\vspace*{-1cm}
\label{table_k}
\end{table}

Recall that $\inf \varnothing =\infty$, so the value of some $\alpha_a$s may be infinite. However, note that the set $ \{ t: t \nmid i, t \nmid j, t \nmid k, t \nmid |i-j|, t \nmid |i-k|, t \nmid |j-k|\}$ defining $\alpha_1$ is always non-empty, and also that $\alpha_1 > 3$. Indeed, at least two of $i,j,k$ have the same parity, so $\alpha_1$ should not divide $2$. Similarly, out of $0,i,j,k$ at least two have the same reminder modulo $3$, so $\alpha_1$ should also not divide $3$. Let $K=\{\alpha_1,\alpha_2,\ldots,\alpha_{14}\}$.

For a pattern $x\pi^i(x)\pi^j(x)\pi^k(x)$, we say that one of the numbers $\alpha_a$ (and its corresponding representation) models an instance $uf^i(u)f^j(u)f^k(u)$ of the pattern in the case when two of the factors $u,f^i(u),f^j(u),f^k(u)$ are equal if and only if the digits associated to the respective factors in the representation of $\alpha_a$ are equal. An infinite word $w$ over some alphabet $\Sigma$ avoids a set $S\subseteq K$ if $w$ contains no instance of the pattern $x \pi^i(x)\pi^j(x)\pi^k(x)$ that is modelled by the parameters of $S$; note that when we discuss about words avoiding a set of parameters, we implicitly assume that the pattern $x \pi^i(x)\pi^j(x)\pi^k(x)$ is fixed. 

Before showing our first results, we need several new notations.

Let $w_1$ and $w_2$ be the digit representation of some $\alpha_{\ell}$, and $\alpha_p$ respectively, with $\ell, p \geq 1$, we say that $w_1$ is a swapped form of $w_2$ if there exists a position $i \leq 4$ such that $w_{1}[i] = w_{2}[i+1]$, and $w_{1}[i+1]= w_{2}[i]$, and $w_{1}[j]=w_{2}[j]$ for all $j \notin \{ i, i+1 \}$. For instance, $0012$ and $0102$ are swapped forms of each others. 

Let $\alpha$ be the digit representation of $x\pi^i(x)\pi^j(x)\pi^k(x)$. We say that $\alpha$ has a prefix square if it starts with $00$, while the other two digits are $1$ and $2$; this is the case for $0012=\alpha_2$. Furthermore, a digit representation has a suffix square if it ends with $22$ and the two other digits are $0$ and $1$; this is the case for $0122=\alpha_5$. We say that $\alpha$ has a gapped square, if it is $0102$, where the $0$s form the gapped square, or if it is $0121$, where the $1$s form the gapped square. We say that $\alpha$ contains a cube if it is $0001$ or $0111$. We say $\alpha$ has two squares if it is $0011$. Finally, $\alpha$ contains gapped cubes if it is $0010$ or $0100$.

Now based on these relations, we define the following collections of sets. The idea behind all these collections is to generate sets of parameters $\alpha_a$s that cannot be avoided and have a minimal cardinality. No matter what will be added to these sets, they will preserve their unavoidability, while erasing something from them will make them avoidable. To obtain these collections we used a computer program and randomly generated some unavoidable sets of parameters of size five. Using the similarities between the instances modelled by these sets, defined in terms of (gapped) squares and cubes occurring in their digit representation, we developed an algorithm to generate more sets of patterns. 

Let $\mathcal{S}_1$ be the collection of sets (each with five elements) that  contain $\alpha_1$ and:
\begin{itemize}
\item one of the $\alpha_a$s whose representation has a prefix or a suffix square, but no gapped cube. That is: $\alpha_2$ or $\alpha_5$.
\item one of the $\alpha_a$s that has a gapped square, but does not have two gapped squares. These are $\alpha_3$ or $\alpha_4$.
\item one of the $\alpha_a$s that contains cubes or two squares: $\alpha_6$ or $\alpha_9$ or $\alpha_{10}$.
\item one of the $\alpha_a$s that contains gapped cubes: $\alpha_7$ or $\alpha_8$.
\end{itemize} 
For example, one possible set from $\mathcal{S}_1$ is $\{\alpha_1,\alpha_2,\alpha_4,\alpha_6,\alpha_7\}$. Note that more sets like this one can be constructed using this scheme, and we should consider all of them, but because of lack of the space, we do not list all the examples here. 

We also have the restriction that if the representations of the squares and gapped squares of a set from $\mathcal{S}_1$ are not swapped form of each other, then the elements of $\mathcal{S}_1$ representing cubes or gapped cubes should have the same digit on all positions of equal digits from the representations of squares and gapped squares. For example, in the case of $0012$ we have that the first and second position contain the same digit and for $0121$ we have that the second and the last position contain the same digit, so our gapped cube should be $0010$ meaning that the first, second and the last position should contain the same digits. 

Let $\mathcal{S}_2$ be the collection of sets (with five elements) that contain $\alpha_1$ and:
\begin{itemize}
\item one of the $\alpha_a$s of the set $\{\alpha_2$, $\alpha_3$, $\alpha_4 \}$, and 
\item one of the $\alpha_a$s of the set $\{\alpha_6$, $\alpha_7$, $\alpha_9 \}$, and
\item both $\alpha_a$s that contain a square in the middle of the word $(\alpha_{12}$ and $\alpha_{13})$.
\end{itemize}
Moreover, we have the restriction that if we choose $\alpha_2$ then $\alpha_7$ should be added to the set. For example, one possible set from $\mathcal{S}_2$ is $\{ \alpha_1, \alpha_2, \alpha_7, \alpha_{12}, \alpha_{13} \}$.


Let $\mathcal{S}_3$ be the collection of sets (with five elements) that contain $\alpha_1$ and $\alpha_{10}$ (the only $\alpha_a$ that has two square factors) as well as:
\begin{itemize}
\item one of the $\alpha_a$s whose representation has a prefix or a suffix square, but no gapped cube. That is: $\alpha_2$ or $\alpha_5$.
\item one of the $\alpha_a$s whose representation has a gapped square, but does not have two gapped squares. That is: $\alpha_3$ or $\alpha_4$.
\item one of the $\alpha_a$s  whose representation contains gapped cubes: $\alpha_7$ or $\alpha_8.$
\end{itemize}
We also have the restriction that if the representation of the squares and gapped squares of a set from $\mathcal{S}_3$ are not swapped form of each other, then its elements representing cubes or gapped cubes should have the same digit on all positions of equal digits from the representations of squares and gapped squares. For example, one possible set from $\mathcal{S}_3$ is  $\{ \alpha_1, \alpha_2, \alpha_4, \alpha_7, \alpha_{10} \}$.

Let $\mathcal{S}_4$ be the collection of sets that contain $\alpha_1,$ $\alpha_2$, $\alpha_7$, and:
\begin{itemize}
\item  one of the $\alpha_a$s whose representations contain cubes ($\alpha_6$ or $\alpha_9$), or  two square  $(\alpha_{10})$ and
\item one of the $\alpha_a$s for whose representation only the first and last digits are equal, and they are different from all other digits $(\alpha_{14})$.
\end{itemize}
One such set is, for example, $\{ \alpha_1, \alpha_2, \alpha_7, \alpha_{10}, \alpha_{14} \}$.

Let $\mathcal{S}_5$ be the collection of sets that contain $\alpha_1$, $\alpha_{12}$, $\alpha_{13}$, and $\alpha_{14}$, as well as one of the $\alpha_a$s  whose representation contains cubes $(\alpha_6$ or $\alpha_9$).
One example is $\{ \alpha_1, \alpha_9,\alpha_{12}, \alpha_{13}, \alpha_{14} \}$.


Let $\mathcal{S}_6$ be the collection of sets that contain $\alpha_1$, $\alpha_{10}$, $\alpha_{13}$, $\alpha_{14}$, and
\begin{itemize}
\item  one of the $\alpha_a$s whose representation contains a cube ($\alpha_6$ or $\alpha_9$), and
\item one of the $\alpha_a$s  whose representation contains a gapped cube $(\alpha_7$ or $\alpha_8)$, and
\item one of the $\alpha_a$s  whose representation contains a gapped square $(\alpha_3$ or $\alpha_4)$.
\end{itemize}
Here we have the restriction that if we have in a set of $\mathcal{S}_6$ the element $\alpha_7$, whose representation has on the first, the second and the last positions the same digit, we should add $\alpha_4$ whose second and last digit are the same. Furthermore, the presence of both $\alpha_4$ and $\alpha_8$ in a set is not permitted.
For example, one possible such set is
$\{ \alpha_1, \alpha_4, \alpha_6, \alpha_7, \alpha_9, \alpha_{10}, \alpha_{13}, \alpha_{14} \}$.

Let $\mathcal{S}_7$ be the collection of sets that contain $\alpha_1$, $\alpha_{12}$, $\alpha_{13}$, and one of the $\alpha_a$s of the set $\{ \alpha_2$, $\alpha_5\}$, and one of the $\alpha_a$s  of the set $\{ \alpha_3$, $\alpha_4 \}$, and one of the $\alpha_a$s of the set $\{ \alpha_7$, $\alpha_8 \}$. Here we have this restriction that the combination of $\{ \alpha_2, \alpha_4 \}$  and  $\{ \alpha_2, \alpha_7 \}$ is not permitted and if the $\alpha_a$s of a set from $\mathcal{S}_7$ whose representation contain squares and gapped squares are not swapped form of each other, then its elements representing cubes or gapped cubes should have the same digit on all positions of equal digits from the representations of $\alpha_a$s that contain squares and gapped squares. For example, one possible such set
$\{ \alpha_1, \alpha_3, \alpha_5, \alpha_8, \alpha_{12}, \alpha_{13} \}$.

Let $\mathcal{S}_8$ be the collection of sets (with six elements) that contain $\alpha_1$ and all elements of the set  $\{ \alpha_3,  \alpha_5, \alpha_7, \alpha_{14} \} $, and
 one of the $\alpha_a$s of the set $\{ \alpha_6$, $\alpha_9 \}$.
One example is $\{  \alpha_1, \alpha_3, \alpha_5, \alpha_7, \alpha_9,  \alpha_{14} \}$.

Let $\mathcal{S}_{9}$ be the collection of sets (with seven elements) that contain $\alpha_1$ and , $\alpha_{10}$:
\begin{itemize}
\item one or two elements of the set  $\{\alpha_2$, $\alpha_5\}$, and
\item all elements of the set $\{\alpha_7$, $\alpha_{14} \}$ or  $\{\alpha_8$, $\alpha_{14} \}$ or one or two elements of the set $\{\alpha_{12}$, $\alpha_{13}$, $\alpha_{14} \}$,.
\item one of the $\alpha_a$s of the set $\{\alpha_3$, $\alpha_4\}$, and
\item one of the $\alpha_a$s of the set $\{\alpha_6$, $\alpha_9\}$, and
\item one of the $\alpha_a$s whose representation has two gapped squares $\alpha_{11}.$
\end{itemize}
Here we have the restriction that if two elements of the set  $\{\alpha_2$, $\alpha_5$, $\alpha_{10}\}$ were selected, one element of the set $\{\alpha_{12}$, $\alpha_{13}$, $\alpha_{14} \}$ should also be chosen, and the other way around. Furthermore, if we choose $\alpha_2$ or $\alpha_5$ or $\alpha_{10}$ as the $\alpha_a$s with squares in a set, then $\alpha_3$, $\alpha_4$, and $\alpha_{14}$ should be selected as gapped squares, respectively, and in the first two conditions, $\alpha_{10}$ and $(\alpha_{12}$ or $\alpha_{13})$, and in the last condition, $(\alpha_{3}$ or $\alpha_{4})$ and $(\alpha_7$ or $\alpha_8$ or $\alpha_{12}$ or $\alpha_{13})$ should be  added to the set. The other restriction is that, if we have the numbers $ \alpha_7$, $\alpha_{14}$, then $\alpha_3$, and if we have the numbers $\alpha_4$, $\alpha_{14}$, then $(\alpha_8$ or $\alpha_{12})$ should be chosen as the gapped squares in a set. In the end, the union of the sets $\{ \alpha_6, \alpha_{10}, \alpha_{13}, \alpha_{14} \}$, and $\{ \alpha_{12}, \alpha_{13} \}$ is not allowed, and the following sets: $\{ \alpha_1, \alpha_3, \alpha_6, \alpha_8, \alpha_{10}, \alpha_{11}, \alpha_{14} \}$, $\{ \alpha_1, \alpha_4, \alpha_5, \alpha_6, \alpha_{10}, \alpha_{12}, \alpha_{14} \}$, $\{ \alpha_1, \alpha_4, \alpha_6, \alpha_7, \alpha_{10}, \alpha_{11}, \alpha_{14} \}$ are also exceptions that should not be considered.
A set fulfilling all the above is $\{ \alpha_1, \alpha_2, \alpha_3, \alpha_6, \alpha_{10}, \alpha_{11}, \alpha_{13} \}$. 


Let $\mathcal{S}_{10}$ be the collection of sets (with eight elements) that contain $\alpha_1$ and (only) one of the following sets: 
\begin{itemize}
\item $\alpha_3$ and one of the sets $\{ \alpha_5, \alpha_6, \alpha_{10}, \alpha_{11}, \alpha_{13}, \alpha_{14} \} $ or $\{ \alpha_5, \alpha_9, \alpha_{10}, \alpha_{11}, \alpha_{13}, \alpha_{14} \} $,  
\item $\{ \alpha_2, \alpha_4, \alpha_{13} \}$ and one of the sets $\{ \alpha_6, \alpha_{10}, \alpha_{11}, \alpha_{14} \} $ or $\{ \alpha_9, \alpha_{10}, \alpha_{11}, \alpha_{14} \} $.
\end{itemize}
For example, one possible such set is
$\{ \alpha_1, \alpha_3, \alpha_5, \alpha_6, \alpha_{10}, \alpha_{11}, \alpha_{13}, \alpha_{14} \}$. 

While the choice of these classes may seem, in a sense, arbitrary, we tried to group together in the same class sets of parameters, according to the common combinatorial features of the elements defining them. The way we obtained these sets is by computer exploration. The idea behind the definition is to generate unavoidable sets of parameters $\alpha_a$ which have a minimal cardinality. We basically started with the sets of size $5$, and tried to extend them one element at a time in order to obtain unavoidable sets. As such, we ensured that removing some element from them leads to an avoidable set of parameters, while, no matter what element we add to them preserves their unavoidability. The reason to start with sets of $5$ parameters is given in the next lemma (shown in the Appendix).

\begin{Lemma}\label{New-Theorem}
Let $K'\subset K$ be any subset of size at most $4$ of $K$. There exists an infinite word $w$ such that $w$ does not contain $4$-powers and if $w$ contains an instance of the pattern $x\pi^i{(x)}\pi^j{(x)}\pi^k{(x)}$ then it can not be modelled by any tuples of the set of patterns $K'$.
\end{Lemma}

The main result of this section is the following theorem.
\begin{theorem}\label{unavoidable_ks}
Given positive integers $i,j,k$ such that $i\neq  j\neq k\neq i$, consider the pattern $p=x\pi^i{(x)}\pi^j{(x)}\pi^k{(x)}$. Let $\sigma =\min \{\max(S)\mid S=S_{\ell}$ for some $\ell=1, \ldots, 10\}$. 
Then $\sigma\geq 4$ and $p$ is unavoidable in $\Sigma_m$, for all $m > \sigma$.
\end{theorem}
\begin{proof}
Because $m \geq  \alpha_1$, we have that for every word $u \in \Sigma^{+}_m$ there exists a morphic permutation $f$ such that every two words of $u, f^{i}(u), f^{j}(u), f^{k}(u)$ are different.  Indeed, we take $f$ to be a permutation such that the orbit of $u[1]$ is a cycle of length $\alpha_1$, which means that the first letters of $u, f^{i}(u), f^{j}(u)$ and $f^{k}(u)$ are pairwise different. Similarly, the fact that $m \geq \alpha_2$ (when $\alpha_2 \neq \infty$) means that for every $u \in \Sigma^{+}_m$ there exists a morphic permutation $f$ such that $f^{k}(u) \neq  u = f^{i}(u) \neq f^{j}(u)\neq f^k(u)$. In this case, we take $f$ to be a permutation such that $ \textbf{ord}_f(u[1])=\alpha_2$. We can derive similar observations for all the $\alpha_a$ parameters involved in the definition of $\sigma$. 


One can check with the aid of a computer, by a backtracking algorithm, that if $m\geq \max(S)+1$, when $S=S_{\ell}$ for some $\ell=1, \ldots, 10$, then $p$ is unavoidable in $\Sigma_m$.  Our computer program tries to construct a word as long as possible by always adding a letter to the current word (obtained by backtracking). This letter is chosen in all possible ways from the letters contained in the word already, or it may also be a new letter, and we just check whether it creates an instance of the pattern as a suffix of the word. Generally, we were not able to check if an arbitrary instance of the pattern is created, due to the complexity of checking all permutations as possible image of $\pi$. But, in most of the cases we need to check, we got the result even when we explicitly allowed $\pi$ to be only a cycle. In the remaining cases, we needed to allow $\pi$ to act as the identity on a symbol, and as a cycle on the rest of the alphabet. This latter case, which was still easy to check, is the reason why we got that $p$ is only avoidable over alphabets of size at least $\sigma+1$ and not already over an alphabet of size $\sigma$. 

For instance, looking at $\mathcal{S}_1$, using a computer program that explores all the possibilities by backtracking we obtained that if $m \geq \max \{ \alpha_1, \alpha_2, \alpha_4, \alpha_6, \alpha_7\}$,  which is at its turn greater or equal to $4$ as $\alpha_1 > 3$, the longest word that does not contain an instance of this pattern, even when constraining $\pi$ to be a cycle, has length
 $36$, and it is $010210210210033001133001133001133000$
(adding new letters to this word does not lead to a longer one). 
On the other hand, we found arbitrarily long words that contain instances of the pattern modelled by $\alpha_1, \alpha_2, \alpha_3, \alpha_6, \alpha_{10}, \alpha_{11}, \alpha_{12}$ when we allow $\pi$ to be replaced only by cycles. However, if we allow $\pi$ to be more general (i.e., only fix one symbol of the alphabet and be a cycle on the rest), we obtain that there are no infinite words that avoid the pattern in this case. So, over alphabets of size $m \geq \max \{ \alpha_1, \alpha_2, \alpha_3, \alpha_6, \alpha_{10}, \alpha_{11}, \alpha_{12}\}+1$ the pattern is unavoidable. 

All the sets of $\alpha_a$s that are used to define $\sigma$ are given in the Appendix. For some of them, we show the longest words that 
do not contain their instances (Table \ref{unavoidablecases1} in the Appendix); these words, as well as words for all the other cases, can be easily found by backtracking. 

Note that by the results (Theorem 4 and Theorem 6) of \cite{TCS2018} it also follows that $\sigma+1\geq 5$. 
 \qed
 
 \end{proof}
\section{Algorithm to generate avoidable cases}
In Lemma \ref{unavoidable_ks}, we proved that given the pattern $x \pi^{i}(x) \pi^j(x) \pi^{k} (x)$, for each $i$, $j$, and $k$, we can compute an upper bound on the minimum size of an alphabet over which the pattern is unavoidable. Now to show that this is the minimum cardinality over which the pattern of size four is unavoidable, we proceed as follows.

Let $K_a$ be the class that contains all nonempty sets of $\alpha_a$ parameters $S'$ such that $S'$ does not include any set $S=S_{\ell}$ for some $\ell=1, \ldots, 10$. In other words, $K_a$ contains all nonempty strict subsets of the sets $S=S_{\ell}$ for some $\ell=1, \ldots, 10$ as well as any other sets of parameters that do not include any of the sets $S=S_{\ell}$ for some $\ell=1, \ldots, 10$. We already know that all subsets of the sets $S=S_{\ell}$ for some $\ell=1, \ldots, 10$ are avoidable. Also, all supersets of the sets $S=S_{\ell}$ for some $\ell=1, \ldots, 10$ are unavoidable (as the sets $S$ already are unavoidable), so we try to show that all the other sets of parameters are avoidable. However, $K_a$ has about 1400 sets of patterns, so checking each of them is hard to be done by pen and paper.

Fortunately, there is an observation we can exploit at this point: all subsets of an avoidable set of parameters is avoidable as well.  For instance, if the set $\{\alpha_1, \alpha_2, \alpha_5, \alpha_6, \alpha_8, \alpha_{14}\}$ can be avoided by a word $\textbf{w}$, then the set $\{ \alpha_1, \alpha_2, \alpha_5 \}$ can also be avoided by $\textbf{w}$. Thus, we can look for the sets of parameters with maximal cardinality that belong to $K_a$ and are avoidable. Clearly, the entire $K$ is unavoidable. However, $K\setminus\{\alpha_1\}$ can be shown to be avoidable. Our approach is implemented in the following algorithmic scheme.

\begin{algorithm}
\caption{Algorithm to generate avoidable cases}\label{avoidable}
\begin{algorithmic}[1]
        \STATE Let $n=10.$ Using the sets $\mathcal{S}_i, (1 \leq i \leq 10)$, generate all sets of $\alpha_a$s of cardinality $n$, that have no unavoidable sets of patterns as subset; show that they are avoidable;
        \STATE {For all $n$ from $9$ down to $4$, generate all sets of cardinality $n$  that have no unavoidable sets of patterns as subset; these sets should not be subsets of the avoidable sets of $\alpha_a$s of cardinality $n+1$ (to avoid generating repetitive avoidable sets of cases generated in the past step); show that they are avoidable.}{} 
\end{algorithmic}
\end{algorithm} 

The following theorem states which sets of $\alpha_a$s can be avoided, according to the algorithm above, concluding thus our approach. It is worth noting that the search space was drastically reduced by our aproach.
\begin{theorem}\label{AvoidableCases}
For each of the following sets there exists an infinite word over an alphabet of size at most $5$, such that if this word contains an instance of $x\pi^i{(x)}\pi^j{(x)}\pi^k{(x)}$ then this instance can not be modelled by an element of the~set.
\begin{align*}
&\{\alpha_{2}, \alpha_{3}, \alpha_{4}, \alpha_{5}, \alpha_{6}, \alpha_{7}, \alpha_{8}, \alpha_{9}, \alpha_{10}, \alpha_{11}, \alpha_{12}, \alpha_{13}, \alpha_{14} \},\\
&\{ \alpha_{1}, \alpha_{2}, \alpha_{3},  \alpha_{4},  \alpha_{5}, \alpha_{6}, \alpha_{9}, \alpha_{10}, \alpha_{11}, \alpha_{14} \}, 
~\{ \alpha_{1}, \alpha_{2}, \alpha_{3}, \alpha_{4}, \alpha_{5}, \alpha_{6}, \alpha_{9}, \alpha_{10}, \alpha_{12},  \alpha_{14} \}, \\
&\{ \alpha_{1}, \alpha_{2}, \alpha_{3}, \alpha_{4}, \alpha_{5},  \alpha_{6},  \alpha_{9}, \alpha_{10}, \alpha_{13},  \alpha_{14} \}, 
~\{ \alpha_{1}, \alpha_{2}, \alpha_{3}, \alpha_{4}, \alpha_{5}, \alpha_{6}, \alpha_{9}, \alpha_{11}, \alpha_{12},  \alpha_{14} \}, \\
&\{ \alpha_{1}, \alpha_{2}, \alpha_{3}, \alpha_{4}, \alpha_{5},  \alpha_{6}, \alpha_{9}, \alpha_{11}, \alpha_{13}, \alpha_{14} \}, 
~\{ \alpha_{1}, \alpha_{2}, \alpha_{3}, \alpha_{4}, \alpha_{5}, \alpha_{7}, \alpha_{8}, \alpha_{11}, \alpha_{12},  \alpha_{14} \}, \\
&\{ \alpha_{1}, \alpha_{2}, \alpha_{3}, \alpha_{4}, \alpha_{5}, \alpha_{7}, \alpha_{8}, \alpha_{11},   \alpha_{13},  \alpha_{14} \}, 
~\{ \alpha_{1}, \alpha_{2}, \alpha_{3}, \alpha_{4}, \alpha_{5}, \alpha_{10}, \alpha_{11}, \alpha_{12}, \alpha_{13}, \alpha_{14} \},\\
&\{ \alpha_{1}, \alpha_{2}, \alpha_{4}, \alpha_{6}, \alpha_{8}, \alpha_{9}, \alpha_{10}, \alpha_{12} \},
~~~~~~~~~~~\{ \alpha_{1}, \alpha_{2}, \alpha_{4}, \alpha_{6}, \alpha_{9}, \alpha_{10}, \alpha_{11}, \alpha_{12} \},\\
&\{ \alpha_{1}, \alpha_{2}, \alpha_{4}, \alpha_{6}, \alpha_{8}, \alpha_{9}, \alpha_{11}, \alpha_{12} , \alpha_{14} \},
~~~~~~\{ \alpha_{1}, \alpha_{2}, \alpha_{4}, \alpha_{8}, \alpha_{10}, \alpha_{11}, \alpha_{12}, \alpha_{13} \},\\
&\{ \alpha_{1}, \alpha_{2}, \alpha_{4}, \alpha_{8}, \alpha_{10}, \alpha_{12}, \alpha_{13}, \alpha_{14} \},
~~~~~~~~~\{ \alpha_{1}, \alpha_{2}, \alpha_{4}, \alpha_{8}, \alpha_{11}, \alpha_{12}, \alpha_{13}, \alpha_{14} \},\\
&\{ \alpha_{1}, \alpha_{2}, \alpha_{4}, \alpha_{6}, \alpha_{8}, \alpha_{9}, \alpha_{10}, \alpha_{11}, \alpha_{13}, \alpha_{14} \},
\{ \alpha_{1}, \alpha_{2}, \alpha_{5}, \alpha_{6}, \alpha_{7}, \alpha_{8}, \alpha_{9}, \alpha_{10}, \alpha_{11}, \alpha_{12}, \alpha_{13}, \alpha_{14} \},  
\\&\{ \alpha_{1}, \alpha_{3}, \alpha_{4}, \alpha_{6},\alpha_{7}, \alpha_{8}, \alpha_{9}, \alpha_{10}, \alpha_{11}, \alpha_{12}, \alpha_{14} \},
\{ \alpha_{1}, \alpha_{3}, \alpha_{4}, \alpha_{6}, \alpha_{7}, \alpha_{8}, \alpha_{9}, \alpha_{10}, \alpha_{11}, \alpha_{13} \}, 
\\&\{ \alpha_{1}, \alpha_{3}, \alpha_{4}, \alpha_{6}, \alpha_{7}, \alpha_{8}, \alpha_{9}, \alpha_{11}, \alpha_{13}, \alpha_{14} \},
\{ \alpha_{1}, \alpha_{2}, \alpha_{4}, \alpha_{6}, \alpha_{8}, \alpha_{9}, \alpha_{10}, \alpha_{11}, \alpha_{13} \},\\
&\{ \alpha_{1}, \alpha_{3}, \alpha_{4}, \alpha_{6}, \alpha_{9}, \alpha_{10}, \alpha_{11}, \alpha_{13},  \alpha_{14} \},
~~~\{ \alpha_{1}, \alpha_{3}, \alpha_{4}, \alpha_{7}, \alpha_{8}, \alpha_{10}, \alpha_{11}, \alpha_{12} , \alpha_{13}\},
\\&\{ \alpha_{1}, \alpha_{3}, \alpha_{4}, \alpha_{7}, \alpha_{8}, \alpha_{10}, \alpha_{11}, \alpha_{13} , \alpha_{14}\},
~~~\{ \alpha_{1}, \alpha_{3}, \alpha_{4}, \alpha_{7}, \alpha_{8}, \alpha_{11}, \alpha_{12},  \alpha_{13} , \alpha_{14}\},
\\&\{ \alpha_{1}, \alpha_{3}, \alpha_{4}, \alpha_{7}, \alpha_{8}, \alpha_{10}, \alpha_{12}, \alpha_{13}, \alpha_{14} \}
\{ \alpha_{1}, \alpha_{3}, \alpha_{5}, \alpha_{6}, \alpha_{7}, \alpha_{9}, \alpha_{10}, \alpha_{12} \},
\\&\{ \alpha_{1}, \alpha_{3}, \alpha_{5}, \alpha_{6}, \alpha_{9}, \alpha_{10}, \alpha_{11}, \alpha_{12} \},
~~~~~~~~~\{ \alpha_{1}, \alpha_{3}, \alpha_{5}, \alpha_{7}, \alpha_{10}, \alpha_{11}, \alpha_{12}, \alpha_{13} \},\\
&\{ \alpha_{1}, \alpha_{3}, \alpha_{5}, \alpha_{7}, \alpha_{10}, \alpha_{12}, \alpha_{13}, \alpha_{14} \},
~~~~~~~~\{ \alpha_{1}, \alpha_{3}, \alpha_{5}, \alpha_{7}, \alpha_{11}, \alpha_{12}, \alpha_{13}, \alpha_{14} \},\\
&\{ \alpha_{1}, \alpha_{3}, \alpha_{5}, \alpha_{6}, \alpha_{7}, \alpha_{9}, \alpha_{10}, \alpha_{11}, \alpha_{13}, \alpha_{14} \}, 
\{ \alpha_{1}, \alpha_{3}, \alpha_{5}, \alpha_{6}, \alpha_{7}, \alpha_{9}, \alpha_{11}, \alpha_{12}, \alpha_{14} \}\\
&\{ \alpha_{1}, \alpha_{3}, \alpha_{7}, \alpha_{10}, \alpha_{11}, \alpha_{12}, \alpha_{13}, \alpha_{14} \},
~~~~~~~\{ \alpha_{1}, \alpha_{4}, \alpha_{6}, \alpha_{8}, \alpha_{9}, \alpha_{10}, \alpha_{12}, \alpha_{14} \},\\
&\{ \alpha_{1}, \alpha_{4}, \alpha_{6}, \alpha_{8}, \alpha_{9}, \alpha_{10}, \alpha_{13}, \alpha_{14} \},
~~~~~~~~~\{ \alpha_{1}, \alpha_{4}, \alpha_{8}, \alpha_{10}, \alpha_{11}, \alpha_{12}, \alpha_{13}, \alpha_{14} \}
\end{align*}

\end{theorem}
\begin{proof}
We only show the statement for the set $T=\{  \alpha_{2}, \alpha_{3}, \alpha_{4}, \alpha_{5}, \alpha_{6}, \alpha_{7}, \alpha_{8}, \alpha_{9},$ $ \alpha_{10},\alpha_{11}, \alpha_{12}, \alpha_{13}, \alpha_{14} \}.$
The other cases can be proved in a similar fashion. Words avoiding them are given in Appendix in Lemmas $2-37$. 

Let ${h}_{\alpha}=\alpha({h})$, where $\alpha : \Sigma_{3}^{*} \to \Sigma_{5}^{*} $ is the morphism defined by \\
\centerline{
$0 \to 0123041203410234, ~~1 \to 0132403124302134, ~~2 \to 0123402134201324.$
}

We show that if ${h}_{\alpha}$ contains an instance of the pattern $x\pi^i{(x)}\pi^j{(x)}\pi^k{(x)}$ then this instance can not be modelled by any tuple of the set $T$. Assume, for the sake of contradiction, that ${h}_{\alpha}$ contains a factor of the form $uf^i(u)f^j(u)f^k(u)$ which can be modelled by any of the $\alpha_a\in T$ (with $f$ morphic permutation). The maximum length of a factor of ${h}_{\alpha}$ that does not contain a full image of any letter of the ternary Thue word under $\alpha$ is $30$. Using a computer program we checked that ${h}_{\alpha}$ has no factor of the form $uf^i(u)f^j(u)f^k(u)$ with $|u|< 31$ which can be modelled by  any of the $\alpha_a\in T$. Further, if $u$ is a word of length $\geq 31$, each of the factors $u, f^i(u), f^j(u), f^k(u)$ contains a full image of a letter of ${h}$. If all these factors contain only the image of $1$ then we get a contradiction, as it would mean that ${h}$ contains a square (either $11$ or a longer square whose image covers $uf^i(u)f^j(u)f^k(u)$, see Appendix). If one of them contains the image of $0$ or $2$ we proceed as follows. Note that the letters $0$ in the image of $0$ under $\alpha$ and the letters $2$  in the image of $2$ occur repeatedly four times, with $3$ symbols between them. So, in one of $u, f^i(u), f^j(u), f^k(u)$, we will have either the image of $0$ under $\alpha$, or the image of $2$ under $\alpha$, and, consequently, four occurrences of $0$, with $3$ other symbols between two consecutive $0$s, or, respectively, four occurrences of $2$, with three other symbols between two consecutive $2$s. Consequently, the four occurrences of $0$ or $2$ should be aligned to four occurrences of another symbol, when considering the alignment of the factors $u, f^i(u), f^j(u), f^k(u)$. Thus, if at least one of the $f^i,f^j,$ or $f^k$ is not the identity, in $uf^i(u)f^j(u)f^k(u)$, based on the repetition of the letters $0$ in the image of $0$, we should have one of the following alignments:
$0123041203410234$ aligned with 
$0412034102340132$ (a contradiction, because this would mean that there is a function mapping $3$ to both $2$ and $3$);
$01230412034102340$ aligned with 
$04120341023401234$ (a contradiction, because this would mean that there is a function mapping $0$ to both $0$ and $4$);
$0123041203410234$ aligned with 
$3240312430213401$ (a contradiction, because this would mean that there is a function mapping $1$ to both $1$ and $4$);
$0123041203410234012$ aligned with 
$2340213420132401230$ (a contradiction, because this would mean that there is a function mapping $0$ to both $2$ and $3$);
$01230412034102340132403124302$ aligned with 
$23402134201324012304120341023$ (a contradiction, because this would mean that there is a function mapping $3$ to both $1$ and $2$);
$01230412034102340$ aligned with 
$23402134201324013$ (a contradiction, because this would mean that there is a function mapping $3$ to both $0$ and $1$);
$01230412034102340$ aligned with 
$21342013240123041$ (a contradiction, because this would mean that there is a function mapping $4$ to both $0$ and $4$). We can apply the same reasoning for the alignments based on the repetition of the letters $2$ in the image of $2$, and get again only contradictions. Therefore, no instance of the pattern is contained in ${h}_{\alpha}$. This concludes our~proof. \qed

\end{proof}

We can now show the main theorem of this paper: 
\begin{theorem}
Given a pattern $p=x\pi^i(x)\pi^j(x)\pi^k(x)$ we can determine effectively the value $\sigma$, such that $p$ is avoidable in $\Sigma_m$ for $m\leq \sigma-1$ and unavoidable in $\Sigma_m$ for $m\geq \sigma+1$. 
\end{theorem}
\begin{proof}
By \cite{DLT2015,TCS2018}, we get that all the unary patterns of size $4$ with permutations are avoidable in $\Sigma_m$ for $m\in\{2,3,4\}$. If $i=j$ or $j=k$ then all the instances of the pattern contain squares, so the pattern is avoidable in $\Sigma_m$ for all $m\geq 3$. If $i=k$, then the pattern is avoidable in $\Sigma_m$, for all $m\geq 3$, according to the results of \cite{MMN12}, where it is shown that $\pi^i(x)\pi^j(x)\pi^i(x)$ is avoidable in such alphabets. 

Let us thus assume that $i\neq j$, $i\neq k$, and $j\neq k$ (which is also the setting of Theorem \ref{unavoidable_ks}). We compute the parameters $\alpha_a$, with $1\leq a\leq 14$, for the given pattern. Then, we consider the sets $S_i$, with $1\leq i\leq 10$, and compute $\sigma=\min\{\max(S)\mid S=S_{\ell}$ for some $\ell=1, \ldots, 10 \}.$ By Theorem \ref{unavoidable_ks} we get that $x\pi^i(x)\pi^j(x)\pi^k(x)$ is unavoidable in $\Sigma_m$, for $m\geq \sigma+1$. Let now $S'=S_{\ell}$ for some $\ell=1, \ldots, 10$ be a set such that $\max(S')=\sigma$. Assume that there exists $\ell\geq 5$ such that $x\pi^i(x)\pi^j(x)\pi^k(x)$ is unavoidable in $\Sigma_\ell$ and $\ell<\sigma$. Let $A_0$ be the set containing all $\alpha_a$ parameters which are at most $\ell$, or, in other words, let $A_0$ be the maximal subset (with respect to inclusion) of $K$ such that if $\alpha\in A_0$ then $\alpha\leq \ell$. Clearly, $A_0$ is either a strict subset of a set $S=S_{\ell}$ for some $\ell=1, \ldots, 10$ or $A_0$ is incomparable to any of the sets of $S$. It cannot include any set $S''=S_{\ell}$ for some $\ell=1, \ldots, 10$ as then $\max(S'')<\max(S')=\min\{\max(S)\mid S=S_{\ell}$ for some $\ell=1, \ldots, 10\}$, a contradiction. Thus $A_0$ is included in one of the sets from the statement of Theorem \ref{AvoidableCases}. Consequently, there exists an infinite word $w$ over a five letter alphabet that avoids $A_0$. In fact, $w$ avoids $A_0$ over all alphabets $\Sigma_m$ such that the instances of $p$ over $\Sigma_m$ correspond only to $\alpha_a$s contained in $A_0$. This means that $w$ avoids $A_0$ in $\Sigma_m$ with $5\leq m\leq \ell$. So, $p$ is avoidable in $\Sigma_\ell$, a contradiction.

In conclusion, the pattern $p=x\pi^i(x)\pi^j(x)\pi^k(x)$ is avoidable in $\Sigma_m$ when $2\leq m< \sigma$. This concludes our proof.
\qed\end{proof}

We get the next corollary, by taking, in the setting of the previous theorem, $\delta=\sigma$, if $x\pi^i(x)\pi^j(x)\pi^k(x)$ is avoidable in $\Sigma_\sigma$, or $\delta=\sigma-1$, otherwise.
\begin{corollary}
Given a pattern $p=x\pi^i(x)\pi^j(x)\pi^k(x)$ there exists $\delta$ a natural number or $+\infty$, such that $p$ is avoidable in $\Sigma_m$ for $m\in \{2,3,\ldots,\delta\}$ and unavoidable in $\nat\setminus \{2,3,\ldots,\delta\}$.
\end{corollary}

\section{Conclusions}
We have shown how to compute, given a pattern $x\pi^i(x)\pi^j(x)\pi^k(x)$, a rather precise approximation of the size of the alphabets where this pattern is avoidable. More importantly, we show that the sizes of these alphabets form an interval of integers. Our results extend the results of \cite{MMN12} and \cite{TCS2018}. The method we used is to explore the number theoretic connections between $i,j,$ and $k$, in relation to the possible orders the permutation $\pi$ may have. This approach follows the initial ideas of \cite{MMN12}, but requires a much more careful and deeper analysis. Essentially, while the relations between $i$ and $j$ in a cubic pattern $x\pi^i(x)\pi^j(x)$ can be modelled with four parameters only, in the case of $x\pi^i(x)\pi^j(x)\pi^k(x)$ we have $14$ such parameters. Exhaustively analysing all the possible relations between these parameters, as it was done in \cite{MMN12}, would take too long, so we devised a less complex way of exploring them. We basically see, on the one hand, which minimal combinations (in the sense of cardinality) of such parameters lead to the conclusion that the pattern is unavoidable in alphabets of large enough size, while also looking for the maximal combinations of the parameters that lead to the conclusion that the pattern is avoidable in alphabets of small enough size. This approach produced a rather large, but still tractable, case analysis.

In order to extend our results to arbitrarily long unary patterns with permutations, we expect that a valid approach would still be based on defining similar sets of parameters and exploring their combinatorial properties. However, it is to be expected that a direct generalization of the ideas above would lead to a number of parameters which grows exponentially with the length of the pattern, hence to a way too complex exploration in the end.

\newpage 
\bibliographystyle{splncs03}
\bibliography{patterns}
\clearpage

\end{document}